\documentclass[12pt]{iopart}
\usepackage{iopams, amsthm, graphicx, subfig, epstopdf}

\theoremstyle{plain}
\newtheorem{theorem}{Theorem}

\newtheorem{lemma}{Lemma}
\theoremstyle{remark}

\theoremstyle{definition}
\newtheorem{definition}{Definition}

\begin{document}
\title[Generalized normal forms of infinitesimal symplectic and contact 
transformations]{Generalized normal forms of infinitesimal symplectic and 
contact transformations in the neighbourhood of a singular point}
\author{A S Vaganyan$^{1,2}$}
\address{Department of Differential Equations, Mathematics and Mechanics 
Faculty, Saint-Petersburg State University, 198504, Universitetsky pr., 28, 
Stary Peterhof, Russia}
\address{Chebyshev Laboratory, 199178, 14th Line 29B, Vasilyevsky Island, 
Saint-Petersburg , Russia}
\ead{armay@yandex.ru}

\begin{abstract}
Definition of generalized normal form for a system of ODEs corresponding to an 
infinitesimal symplectic or contact transformation near a singular point, with 
an arbitrary polynomial unperturbed part, and a method of its finding are 
introduced. Applicability of the introduced method to studying the critical 
phenomena in non-ideal media is shown. As examples, generalized normal forms 
for the equations of state of a mixture of non-ideal gases and non-ideal 
multicomponent plasma are considered within the framework of perturbation 
theory. In particular, it is shown that the lowest order perturbation effects 
in the Debye-H\"uckel hydrogen plasmas are classified by only three constant 
parameters.
\end{abstract}
\maketitle

\section{Introduction}

In the introduction, some necessary facts from symplectic and contact geometry 
are given. For a complete overview see \cite[Ch~1, \S\S~6,\,7]{kobayasi} and 
references therein.

\bigskip
A \textit{symplectic form} on a manifold $M$ of even 
dimension $s=2\,n$ is a closed 2-form $\omega$ of maximal rank, i.~e. 
$\rmd\omega=0$ and $\omega^n\ne0$ everywhere on $M.$ There exists a local 
coordinate system $z_1,\,\ldots,\,z_{2n}$ for which 
$\omega=\sum_{i=1}^n \rmd z_{n+i}\wedge\rmd z_i$ (see \cite[p~154]{kobayasi}).
Such coordinates are called \textit{canonical coordinates}. 

A transformation $f$ (or an infinitesimal transformation $X,$ respectively) of 
$M$ is said to be \textit{symplectic} if $f^*\,\omega=\omega$ 
$(\mathrm L_X\,\omega=0),$ where $f^*$ denotes the linear mapping transposed 
to $\rmd f,$ and $\mathrm L_X$ is the Lie derivative along the vector field 
$X.$ 

In canonical coordinates, an infinitesimal symplectic transformation is given 
by
\begin{equation}\label{Eq01}
X_H=\sum_{i=1}^n 
\frac{\partial H}{\partial z_{n+i}}\,\frac{\partial}{\partial z_i}-
\frac{\partial H}{\partial z_i}\,\frac{\partial}{\partial z_{n+i}},
\end{equation}
where $H$ is an arbitrary differentiable function. $H$ is called  
\textit{the Hamiltonian}, and infinitesimal transformation \eref{Eq01} is 
called a \textit{Hamiltonian vector field}. The corresponding system of ODEs 
is called a \textit{Hamiltonian system}.

Defining the \textit{Poisson bracket} of two functions $\Phi$ and $\Psi$ by  
$X_{\{\Phi,\,\Psi\}}=[X_\Phi,\,X_\Psi],$ in canonical coordinates we have the 
classical formula
\begin{equation*}
\{\Phi,\,\Psi\}=\sum_{i=1}^n 
\frac{\partial\Phi}{\partial z_i}\,\frac{\partial\Psi}{\partial z_{n+i}}-
\frac{\partial\Phi}{\partial z_{n+i}}\,\frac{\partial\Psi}{\partial z_i}.
\end{equation*}

\bigskip 
A \textit{contact form} on a manifold $M$ of odd
dimension $s=2\,n+1$ is an open cover $\{U_\alpha\}$ of $M$ together with a 
system of 1-forms $\omega_\alpha$ of maximal rank, i.~e. 
$\omega_\alpha\wedge(\rmd\omega_\alpha)^n\ne0$ everywhere on $U_\alpha,$ such 
that $\omega_\alpha=\varphi_{\alpha\beta}\,\omega_\beta$ on 
$U_\alpha\cap U_\beta$ for some functions $\varphi_{\alpha\beta}\ne0.$  

For any point $p\in U_\alpha,$ there exists a coordinate neighbourhood 
$U_p\subset U_\alpha$ with coordinates $z_1,\,\ldots,\,z_{2n+1}$ in which 
$\omega_\alpha=\rmd z_{2n+1}-\sum_{i=1}^n z_{n+i}\,\rmd z_i$ (see 
\cite[p~150]{kobayasi}). Just as in the symplectic case, we call such 
coordinates \textit{canonical}. 

A transformation $f$ of $M$ is said to be \textit{contact} if for all 
$\alpha,\,\beta,$ $f^*\,\omega_\alpha=\Lambda_{\alpha\beta}\,\omega_\beta,$ 
where $\Lambda_{\alpha\beta}\ne0$ everywhere on 
$f^{-1}(U_\alpha)\cap U_\beta.$ Accordingly, an infinitesimal transformation 
$X$ is called \textit{contact} if 
$\mathrm L_X\,\omega_\alpha=\lambda_\alpha\,\omega_\alpha$ for some functions 
$\lambda_\alpha.$

In canonical coordinates, an infinitesimal contact transformation is given by
\begin{eqnarray}\fl
X_H=\sum_{i=1}^n\Big[-\frac{\partial H}{\partial z_{n+i}}\,
\frac{\partial}{\partial z_i}+\Big(\frac{\partial H}{\partial z_i}+z_{n+i}\,
\frac{\partial H}{\partial z_{2n+1}}\Big)\,
\frac{\partial}{\partial z_{n+i}}\Big]+\nonumber\\+
\Big(H-\sum_{j=1}^n z_{n+j}\,\frac{\partial H}{\partial z_{n+j}}\Big)\,
\frac{\partial}{\partial z_{2n+1}},\label{Eq02}
\end{eqnarray}
where $H$ is an arbitrary differentiable function. The function $H$ is 
sometimes called a contact Hamiltonian, or simply \textit{the Hamiltonian}. 

In the contact case, the Poisson bracket in canonical coordinates has the view
\begin{eqnarray}\fl
\{\Phi,\,\Psi\}=\Phi\,\frac{\partial\Psi}{\partial z_{2n+1}}-
\Psi\,\frac{\partial\Phi}{\partial z_{2n+1}}+\nonumber\\+\sum_{i=1}^n
\Big(\frac{\partial\Phi}{\partial z_i}+
z_{n+i}\,\frac{\partial\Phi}{\partial z_{2n+1}}\Big)\,
\frac{\partial\Psi}{\partial z_{n+i}}-
\Big(\frac{\partial\Psi}{\partial z_i}+
z_{n+i}\,\frac{\partial\Psi}{\partial z_{2n+1}}\Big)\,
\frac{\partial\Phi}{\partial z_{n+i}}.\label{Eq03}
\end{eqnarray}

\bigskip
From now on, only functions that are analytic in some neighbourhood $U$ of the 
origin in $\mathbb R^s$ are considered. We denote the algebra of analytic 
functions in $U$ by $\mathbb R_U(z),$ and the algebra of polynomial functions 
in $z_1,\ldots,\,z_s$ by $\mathbb R[z].$ 

Diffeomorphisms of $\mathbb R^s$ onto itself of the form $\exp(X_P)$ with
$P\in\mathbb R[z]$ generate the subgroup $\mathfrak A_s$ in the group of
symplectic (if $s=2\,n),$ or contact (if $s=2\,n+1),$ transformations of 
$\mathbb R^s.$ The algebra of infinitesimal symplectic, or contact, 
transformations is closed under the action of this subgroup, and the 
Hamiltonian functions are transformed according to the law 
\begin{equation*}
H\,\rightarrow\,\exp(\widehat P)(H),
\end{equation*}
where $\widehat P(\,\cdot\,)=\{P,\,\cdot\,\}.$

Let the origin be a \textit{singular point} of the vector field $X_H,$ i.~e. 
$X_H|_{z=0}=0,$ and let the corresponding Hamiltonian have the view
\begin{equation}\label{Eq04}
H=H_0+H_1\qquad\big(H_0\in\mathbb R[z],\ H_1\in\mathbb R_U(z)\big).
\end{equation}
We call $H_0$ the \textit{unperturbed part}, and $H_1$ the 
\textit{perturbation} of the Hamiltonian $H.$

In the first four sections of the present article, the general problem of 
bringing the perturbation of $H$ into the simplest form (\textit{normal form}) 
by transformations from the group $\mathfrak A_s$ that do not change the 
unperturbed part of $H$ is being solved. Such problems often arise in studying 
Hamiltonian and nearly Hamiltonian systems. As a result, there are many 
different ad hoc methods and definitions of normal forms (see, for instance, 
\cite{bruno, meyer, belitsky}) that require the unperturbed Hamiltonian to 
satisfy certain conditions such as quadraticity, homogeneity, etc. But such 
requirements can be excessive in some physical problems. The generalized 
normal form introduced in this article is free of such restrictions. It is 
worth noting that the contact case was earlier considered only in the 
simplest, generic, situation. However, it turns out that real physical 
problems, for instance thermodynamics of plasmas, often refer to the 
non-generic cases. In section~5 we introduce some applications of the method 
of generalized normal forms to thermodynamics of non-ideal media.

\section{Basic definitions}

As usual, the sets of positive and nonnegative integers will be denoted by 
$\mathbb N$ and $\mathbb Z_+$ respectively.

Let $n\in\mathbb N$ and $s=2\,n,$ or $2\,n+1,$ be fixed. 

\begin{definition} 
A vector $\gamma=(\gamma_1,\ldots,\,\gamma_s)\in\mathbb N^s$ is called a 
\textit{weight} of the variable $z=(z_1,\ldots,\,z_s)$ if 
$\mathrm{GCD}\,(\gamma_1,\ldots,\,\gamma_s)=1.$ If, moreover, for some natural
$\sigma\ge2,$ $\gamma_1+\gamma_{n+1}=\ldots=\gamma_n+\gamma_{2n}=\sigma,$ and 
in the case of odd $s,$ $\gamma_{2n+1}=\sigma,$ we say that $\gamma$ is 
a \textit{canonical weight}.
\end{definition}

\begin{definition} 
Given a power series (or a polynomial) 
$P=\sum_{\nu\in\mathbb Z_+^s} a_\nu\,z^\nu,$ the \textit{generalized order} 
(the \textit{generalized degree}) of $P$ with weight $\gamma$ is the least 
(respectively, the greatest) of numbers 
$\nu\cdot\gamma=\nu_1\,\gamma_1+\ldots+\nu_s\,\gamma_s$ such that $a_\nu\ne0$ 
if $P\not\equiv0,$ and $\infty$ (or $0,$ respectively) otherwise.
\end{definition}

\begin{definition} 
A polynomial $P\in\mathbb R[z]$ is said to be \textit{quasi-homogeneous} with 
weight $\gamma$ if either its generalized order and generalized degree with 
weight $\gamma$ are equal or $P$ is identically zero.
\end{definition}

\begin{definition} 
Given an arbitrary power series $P=\sum_{\nu\in\mathbb Z_+^s} a_\nu\,z^\nu,$ 
define its $[k]$-jet with weight $\gamma$ by 
$\mathrm{J}_\gamma^{[k]}(P)=\sum_{\nu\cdot\gamma=0}^k a_\nu\,z^\nu.$
\end{definition}

Denote the generalized order of $P$ with weight $\gamma$ by 
$\mathrm{ord}_\gamma\,P,$ the generalized degree by $\deg_\gamma P,$ and the 
vector space of all $[k]$-jets with weight $\gamma$ by 
$\mathfrak J_\gamma^{[k]}.$

Consider a power series $H(z)=H_0(z)+H_1(z)$ with the unperturbed part 
$H_0\in\mathbb R[z]$ and the perturbation $H_1\in\mathbb R_U(z)$ such 
that for some fixed canonical weight $\gamma,$
\begin{equation}\label{Eq05}
\mathrm{ord}_\gamma\,H_1>\deg_\gamma H_0\ge\mathrm{ord}_\gamma\,H_0. 
\end{equation}

For brevity, denote $\delta=\mathrm{ord}_\gamma\,H_1-
\mathrm{ord}_\gamma\,H_0$ and $\Delta=\mathrm{ord}_\gamma\,H_1.$
 
Define an inner product on polynomials by
\begin{equation}\label{Eq06}
\langle\!\langle P,\,Q\rangle\!\rangle=P(\partial)\,Q(z)|_{z=0}\qquad
\big(P,\,Q\in\mathbb R[z],\ 
\partial=(\partial/\partial z_1,\ldots,\,\partial/\partial z_s)\big),
\end{equation}
and fix some linear order $\succ$ on the set of monomials in $z.$

For any integer $d\ge\Delta,$ define a vector space
$\mathfrak N_d\subset\mathfrak J_\gamma^{[d+\delta-1]}$ as follows:
\begin{equation}\label{Eq07}
\mathfrak N_d=\Big\{\mathrm{J}_\gamma^{[d+\delta-1]}\big(\{F,\,H_0\}\big):\ 
F\in\mathbb R[z],\ \mathrm{ord}_\gamma\,F\ge d+\delta-\Delta+\sigma\Big\},
\end{equation} 
and denote the set of leading monomials of the elements of $\mathfrak N_d$ 
with respect to $\succ$ by $\mathrm{LM}\,(\mathfrak N_d).$

\begin{lemma} 
For any integer $d\ge\Delta,$ there exists a basis $\mathfrak G_d=\{G_i\}$ of 
$\mathfrak N_d$ such that  $\langle\!\langle G_i,\,
\mathrm{LM}\,(G_j)\rangle\!\rangle=\delta_{ij},$ where $\delta_{ij}$ denotes 
the Kronecker delta. Herewith, the sets $\mathrm{LM}\,(\mathfrak G_d)$ and 
$\mathrm{LM}\,(\mathfrak N_d)$ are the same.
\end{lemma}

\begin{proof}
The proof is by induction. Let $\{B_i\}$ be an arbitrary basis for 
$\mathfrak N_d.$ Setting $G_1=B_1/\langle\!\langle B_1,\,
\mathrm{LM}\,(B_1)\rangle\!\rangle$ gives
$\mathrm{LM}\,(G_1)=\mathrm{LM}\,(B_1)$ and
$\langle\!\langle G_1,\,\mathrm{LM}\,(G_1)\rangle\!\rangle=1.$

Suppose that for some $k\in\mathbb N$ and for all $i,\,j\le k,$ we have
$\langle\!\langle B_i,\,\mathrm{LM}\,(B_j)\rangle\!\rangle=\delta_{ij}.$  
Denote $G'_{k+1}=B_{k+1}-\sum_{i=1}^k 
\langle\!\langle B_{k+1},\,\mathrm{LM}\,(B_i)\rangle\!\rangle B_i,$ so that 
$\langle\!\langle G'_{k+1},\,\mathrm{LM}\,(B_i)\rangle\!\rangle=0.$ Then 
setting $G_{k+1}=G'_{k+1}/\langle\!\langle G'_{k+1},\,
\mathrm{LM}\,(G'_{k+1})\rangle\!\rangle$ and $G_i=B_i-\langle\!\langle B_i,\,
\mathrm{LM}\,(G_{k+1})\rangle\!\rangle\,G_{k+1}$ for all $i\le k$ gives 
$\mathrm{LM}\,(G_i)=\mathrm{LM}\,(B_i),$ and hence $\langle\!\langle G_i,\,
\mathrm{LM}\,(G_j)\rangle\!\rangle=\delta_{ij}$ for all $i,\,j\le k+1,$ which 
completes the induction and proves the first statement of the lemma.

Since $\mathfrak N_d$ is finite-dimensional, the second statement follows at 
once from the fact that all of the elements of $\mathfrak G_d$ have different 
leading monomials.
\end{proof}

\begin{definition}\label{def5} 
Define a \textit{minimal resonant set of index $d$} as a complement of 
$\mathrm{LM}\,(\mathfrak N_d)$ in the set of monomials of generalized degree 
$k,$ where $d\le k<d+\delta.$    	
\end{definition}

Denote the vector space spanned by a minimal resonant set of index $d$ by 
$\mathfrak R_d.$ Then given a basis $\mathfrak G_d$ for $\mathfrak N_d$ as in 
Lemma~1, the formula 
\begin{equation}\label{Eq08}
\pi_d(P)=P-\sum_i\langle\!\langle P,\,\mathrm{LM}\,(G_i)\rangle\!\rangle\,G_i
\end{equation}
defines a projection of the vector space of $[d+\delta-1]$-jets of  
generalized order greater than or equal to $d$ onto $\mathfrak R_d.$ 

\section{Normal form theorem}

\begin{theorem} 
For any integer $N\ge\Delta,$ there exists a transformation 
from $\mathfrak A_s$ that brings $H$ into the form
\begin{equation}\label{Eq09}
\widetilde H=H_0+\sum_{k=0}^{\big[(N-\Delta)/\delta\big]}R_k+S,
\end{equation} 
where $R_k\in\mathfrak R_{\Delta+k\delta}$ and $\mathrm{ord}_\gamma\,S>N.$ 
\end{theorem}
\begin{proof}
Let $F\in\mathbb R[z]$ and $\mathrm{ord}_\gamma\,F\ge\delta+\sigma.$ Then the 
transformation $\exp(X_F)\in\mathfrak A_s$ brings $H$ into the form
$$\exp(\widehat F)(H)=H_0+H_1+\{F,H_0\}+\ldots,$$
where $\mathrm{ord}_\gamma\,\big(\{F,\,H_0\}\big)\ge\Delta,$ and the dots 
denote terms of generalized order greater than or equal to $\Delta+\delta.$

More generally, if $\mathrm{ord}_\gamma\,F\ge k\,\delta+\sigma$ for some 
$k\in\mathbb N,$ then in the same formula 
$\mathrm{ord}_\gamma\,\big(\{F,\,H_0\}\big)\ge\Delta+(k-1)\,\delta,$ and the 
generalized order of terms denoted by dots is greater than or equal to  
$\Delta+k\,\delta.$

Suppose that $H=H_0+\sum_{k=0}^{m-1}R_k+S_m,$ where 
$R_k\in\mathfrak R_{\Delta+k\delta},$ 
$\mathrm{ord}_\gamma\,S_m\ge\Delta+m\,\delta,$ for some $m\in\mathbb Z_+.$ The 
proof is by induction on $m$ with $S_0=H_1$ for $m=0.$ 

Write $S_m$ as $S_m=P_m+Q_m,$ where
$P_m=\mathrm{J}_\gamma^{[\Delta+(m+1)\,\delta-1]}(S_m).$ Since 
$P_m-\pi_{\Delta+m\delta}(P_m)\in\mathfrak N_{\Delta+m\delta},$ where 
$\pi_{\Delta+m\delta}$ and $\mathfrak N_{\Delta+m\delta}$ are as defined in 
\eref{Eq07}--\eref{Eq08}, there exists a polynomial $F_m$ of generalized order 
greater than or equal to $(m+1)\,\delta+\sigma$ such that 
$P_m+\mathrm{J}_\gamma^{[\Delta+(m+1)\,\delta-1]}(\{F_m,H_0\})=
\pi_{\Delta+m\delta}(P_m).$ Denote $R_m=\pi_{\Delta+m\delta}(P_m).$ Then the 
transformation $\exp(X_{F_m})\in\mathfrak A_s$ brings $H$ into the form 
$H_0+\sum_{k=0}^{m}R_k+S_{m+1},$ where
$\mathrm{ord}_\gamma\,S_{m+1}\ge\Delta+(m+1)\,\delta.$ 

By induction, $m$ can be made arbitrarily large, and in particular, equal to 
$[(N-\Delta)/\delta],$ so that $\mathrm{ord}_\gamma\,S_{m+1}\ge\Delta+
\big([(N-\Delta)/\delta]+1\big)\,\delta>N.$ 
Since at each step the normalizing transformation has the form 
$\exp(X_{F_m})\in\mathfrak A_s,$ and the number of steps is finite, 
the composition of these transformations lies in $\mathfrak A_s.$ 
\end{proof}

\begin{definition} 
We say a Hamiltonian $H$ with the unperturbed part $H_0,$ as well as the 
corresponding vector field, to be in generalized normal form up to generalized 
degree $N$ if it has the form \eref{Eq09}.
\end{definition}

\section{The case of quasi-homogeneous unperturbed part}

In case where the unperturbed Hamiltonian is quasi-homogeneous for some 
canonical weight, there is an alternative way of describing generalized normal 
forms in terms of what we call the ''resonant sets''. This method requires no 
linear order on monomials and is stated as follows.

Denote the vector space of quasi-homogeneous polynomials in $z$ of generalized 
degree $k\in\mathbb Z_+$ with weight $\gamma$ over $\mathbb R$ by 
$\mathbb R_\gamma^{[k]}.$

Let $H_0=H_\gamma^{[\chi]}\in\mathbb R_\gamma^{[\chi]},$ where $\gamma$ is a 
canonical weight, and $\chi\ge\sigma.$ Then for each $k\in\mathbb Z_+,$ the 
linear operator $\widehat H_\gamma^{[\chi]}=\{H_\gamma^{[\chi]},\,\cdot\,\}$ 
maps $\mathbb R_\gamma^{[k]}$ to $\mathbb R_\gamma^{[k+\chi-\sigma]},$ and its 
conjugate with respect to the inner product \eref{Eq06} has the form
\begin{equation*}
\left.\widehat{H}_\gamma^{[\chi]}\right.^*=
\sum_{i=1}^n
z_{n+i}\,\Big(\frac{\partial H_\gamma^{[\chi]}}{\partial z_i}\Big)^*-
z_i\,\Big(\frac{\partial H_\gamma^{[\chi]}}{\partial z_{n+i}}\Big)^*,
\end{equation*}
in the symplectic case, and
\begin{eqnarray*}\fl
\left.\widehat{H}_\gamma^{[\chi]}\right.^*=
z_{2n+1}\,\left.H_\gamma^{[\chi]}\right.^*-
\Big(\frac{\partial H_\gamma^{[\chi]}}{\partial z_{2n+1}}\Big)^*+\\+
\sum_{i=1}^n z_{n+i}\,\Big(\frac{\partial H_\gamma^{[\chi]}}{\partial z_i}+
z_{n+i}\,\frac{\partial H_\gamma^{[\chi]}}{\partial z_{2n+1}}\Big)^*-
\Big(z_i+z_{2n+1}\,\frac{\partial}{\partial z_{n+i}}\Big)\,
\Big(\frac{\partial H_\gamma^{[\chi]}}{\partial z_{n+i}}\Big)^*
\end{eqnarray*}
in the contact case. Here $P^*=P(\partial)$ for each $P\in\mathbb R[z].$

\begin{definition}
We call the equation
\begin{equation}\label{Eq10}
\left.\widehat{H}_\gamma^{[\chi]}\right.^*(P)=0
\end{equation} 
for $P\in\mathbb R[z]$ the \textit{resonance equation}, and its solutions  
\textit{resonant polynomials}.  
\end{definition} 

Denote the vector space of resonant polynomials by $\mathfrak R_\gamma.$

By properties of the inner product \eref{Eq06}, it follows that 
$\mathfrak R_\gamma$ splits into direct sum of the orthogonal subspaces
$\mathfrak R_\gamma^{[k]}$ spanned by quasi-homogeneous resonant polynomials 
of generalized degree $k\in\mathbb Z_+.$

\begin{definition}\label{def6}
We call a set of quasi-homogeneous polynomials $\mathfrak S_\gamma^{[m]}=
\{S_j\}_{j=1}^{s_m}$ a resonant set of generalized degree $m$ with weight 
$\gamma$ if for some basis $\{R_i\}_{i=1}^{s_m}$ of 
$\mathfrak{R}_\gamma^{[m]},$ $\det(\{\langle\!\langle R_i,\,
S_j\rangle\!\rangle\}_{i,j=1}^{s_m})\ne0.$ Given a resonant set 
$\mathfrak S_\gamma^{[m]}$ for each generalized degree $m\in\mathbb Z_+,$ we 
call their union, $\mathfrak S_\gamma=\bigcup_{m=0}^\infty
\mathfrak S_\gamma^{[m]},$ a \textit{resonant set}. If, moreover, 
$\mathfrak S_\gamma$ consists only of monomials, we say that it is a 
\textit{minimal resonant set}.
\end{definition}

Clearly, the definition of a resonant set does not depend on the choice of 
basis for $\mathfrak{R}_\gamma^{[k]}.$ Note that in the case of a 
quasi-homogeneous unperturbed part, a minimal resonant set in the sense of 
Definition~\ref{def5} is also a minimal resonant set in the sense of 
Definition~\ref{def6}, and vice versa.

\begin{theorem} 
For each integer $N\ge\Delta$ and arbitrarily chosen by $H_\gamma^{[\chi]}$ 
resonant set $\mathfrak S_\gamma,$ there exists a transformation from 
$\mathfrak A_s$ that brings $H$ to the form
\begin{equation*}
\widetilde H=H_0+\sum_{k=\Delta}^{N} R_\gamma^{[k]}+S,
\end{equation*}
where $R_\gamma^{[k]}\in\mathrm{span}(\mathfrak S_\gamma^{[k]})$ and 
$\mathrm{ord}_\gamma\,S>N.$ 
\end{theorem}

\begin{proof}
The theorem follows from Definition~\ref{def6} and the Fredholm alternative 
for finite-dimensional spaces. Detailed proof in the symplectic case (in the 
contact case, the proof is exactly the same) is given in \cite{BV}. 
\end{proof}

For the symplectic case, a close result to Theorem~2 can be found in 
\cite{belitsky}. However, the normal form equations given there are very 
complicated, since the unperturbed part at each step of normalization involves 
normalized terms of the perturbation (for further information, see 
\cite{belitsky77}). The quasi-homogeneous symplectic case 
was first analyzed in \cite{BV}, where the canonicity condition on the weight 
was stated and definition of a resonant set was introduced. 

\section{Applications of generalized normal forms in thermodynamics}

\subsection{Thermal equations of state}

In thermodynamics, relations between macroscopic parameters are described by 
the equations of state. Equations that provide a relationship between the 
temperature~$T,$ the volume concentrations $n_i$ $(i=\overline{1,\,k})$ of  
the system components, where $k$ denotes the number of components, the 
pressure $P,$ and other generalized thermodynamic forces, are called 
\textit{thermal equations of state}. For example, 
\begin{eqnarray}\label{Eq11}
P=T\,\sum_{i=1}^k n_i,\\ \label{Eq12}
P=T\,\sum_{i=1}^k n_i-
\frac{\sqrt{\pi}}{3\,T^{1/2}}\,\Big(\sum_{i=1}^k n_i\,q_i^2\Big)^{3/2}
\end{eqnarray}
are, respectively, thermal equations of state of the ideal gas and the 
classical Debye-H\"uckel plasma written in ESU CGS units, where $T$ is 
considered in $\mathrm{erg}$ (see, for instance, \cite[pp~151,\,282]{landau}). 
Here, $q_i$ denotes the electric charge of a particle of $i$-th type.

There are two approaches to describing thermodynamical properties of non-ideal 
media, the rigorous one treating media as a system of nuclei and electrons 
interacting by Coulomb's law, and the \textit{semi-empirical} one that 
consists in constructing qualitative models, in which the general form of 
functional relationships is established on theoretical arguments, and the 
numerical coefficients are determined experimentally \cite[s~8]{fortov}. 
Method of generalized normal forms allows to reduce significantly the number 
of unknown coefficients in such qualitative models of non-ideal media.

Consider a thermal equation of state of the form
\begin{equation}\label{Eq13}
H(P,\,T,\,n)=0,\qquad n=(n_1,\,\ldots,\,n_k).
\end{equation} 
Using the \textit{Gibbs-Duhem equation} (see, for instance, 
\cite[p~322]{krichevsky})
\begin{equation*}
\rmd P-s\,\rmd T-\sum_{i=1}^k\,n_i\,\rmd\mu_i=0,
\end{equation*}
where $s$ denotes the entropy density, and $\mu_i$ the chemical potential of 
the $i$-th type of particles, \eref{Eq13} is seen to be a first order 
partial differential equation for $P.$ Its characteristic equations are 
determined by infinitesimal contact transformation \eref{Eq02}:
\begin{equation}\label{Eq14}
\cases{\dot P=H-\sum_i n_i\,\frac{\partial H}{\partial n_i}\\
\dot n_i=n_i\,\frac{\partial H}{\partial P}\\
\dot T=0},\quad
\cases{\dot \mu_i=-\frac{\partial H}{\partial n_i}\\
\dot s=\frac{\partial H}{\partial T}+s\,\frac{\partial H}{\partial P}},
\end{equation}
where subsystem~$(\ref{Eq14}_1)$ contains only variables that enter 
\eref{Eq13}. Here, the dot denotes differentiating with respect to some 
parameter that does not have a physical sense itself. Note that the quantities 
$T,\,n_j/\sum_i n_i,\,H/\sum_i n_i$ are integrals of \eref{Eq14}, and the 
surface given by \eref{Eq13} is integral for this system.

Suppose that $H$ can be represented in the form \eref{Eq04}--\eref{Eq05}, 
where $z_i=\mu_i,$ $z_{k+1}=T,$ $z_{k+i+1}=n_i,$ $z_{2(k+1)}=s,$ 
$z_{2(k+1)+1}=P$ $(i=\overline{1,\,k}).$ Then the above theory is applicable, 
so by Theorem~1, $H$ can be put into normal form $\widetilde H$ up to a 
desired degree by a contact transformation, with the surface $H=0$ transformed 
to $\widetilde H=0.$

By definition, variables $s$ and $\mu_i$ do not enter a thermal equation of 
state. So if we do not want them to appear in $\widetilde H,$ it is 
necessary to consider the transformations that depend only on $P,\,T,$ and 
$n_i.$ This follows from the form of the Poisson bracket \eref{Eq03}. 

\subsection{Perturbation theory for ideal gases}

Equation of state of a mixture of non-ideal gases is usually written in the 
form of a \textit{virial expansion}, also called the \textit{Mayer cluster 
expansion}:
\begin{equation}\label{Eq15}
P=T\,\sum_{i=1}^k n_i+\sum_{|m|\ge2} B_m\,n^m,\qquad m=(m_1,\,\ldots,\,m_k).
\end{equation}
$B_m$ are called \textit{virial coefficients} and characterized by  
contributions of the interactions within $|m|$-particle groups, or 
\textit{clusters}, of molecules of gas.

Suppose that $B_m$ are analytic functions of temperature near the point $T=0.$ 
Then the following theorem holds.

\begin{theorem} 
For each integer $M\ge2,$ all the virial coefficients $B_m$ with $|m|\le M$ in 
expansion \eref{Eq15} can be eliminated by a contact transformation.
\end{theorem}
\begin{proof} 
In generalized degree $2$ with weight $\gamma=(1,\,\ldots,\,1,\,2),$ where $2$ 
is the weight of $P,$ equation of state \eref{Eq15} can be written as 
\begin{equation*}
P=T\,\sum_{i=1}^k n_i+\sum_{i,j=1}^k \frac{b_{ij}\,n_i\,n_j}{2}.
\end{equation*}
Here $b$ is a constant symmetric matrix. Clearly, the contact transformation 
$\widetilde P=P-\sum_{i,j=1}^k b_{ij}\,n_i\,n_j/2,$ 
$\widetilde \mu_i=\mu_i-\sum_{j=1}^k b_{ij}\,n_j$ brings it into the form 
\eref{Eq11}. Thus, the part of the second virial coefficients that does not 
depend on temperature can always be eliminated by a contact transformation.

Let us prove the statement for terms of the perturbation of generalized degree 
greater than $2.$ For this purpose write out the resonance 
equation \eref{Eq10} for the unperturbed Hamiltonian 
$H_\gamma^{[2]}=P-T\,\sum_{i=1}^k n_i$ corresponding to \eref{Eq11}, taking 
into account that the perturbation depends only on $P,\,T,$ and $n_i:$
\begin{equation*}
P\,\frac{\partial R}{\partial P}+
\sum_{i=1}^k n_i\,\frac{\partial R}{\partial n_i}-R=0\qquad
\big(R\in\mathbb R[P,\,T,\,n_i]\big).
\end{equation*}
Monomials of the form $P\,T^l$ and $n_i\,T^l$ make up a basis of solutions 
of this equation, and hence by Theorem~2, normal form of the perturbation. In 
particular, terms of degree higher than one in $n_i$ are absent in normal form.
\end{proof}

Theorem~3 agrees with a well known fact that the media described by virial 
expansions (for example, the Van der Waals gas) permit continious transitions 
between any two states without ever separating into two different phases 
\cite[pp~298,\,299]{landau}. Apparently, for the same reasons as claimed in 
Theorem~3, the virial expansion turned out to be useless for describing 
critical phenomena, such as the gas-liquid-solid triple point 
(for further discussion, see \cite[pp~105--110]{feynman}).

\subsection{Perturbation theory for classical Debye-H\"uckel plasmas}

Equation of state \eref{Eq12} of $k$-component plasma holds under 
assumptions of electrical neutrality and smallness of the Coulomb interaction 
in comparison to the kinetic energy \cite[pp~279,\,280]{landau}:
\begin{equation}\label{Eq16}
\sum_{i=1}^k n_i\,q_i=0,\quad
\sum_{i=1}^k n_i\ll\bigg(\frac{T}{\langle q^2\rangle}\bigg)^3.
\end{equation}

The unperturbed Hamiltonian corresponding to \eref{Eq12} can be chosen as
\begin{equation}\label{Eq17}
H_0=T\,\Big(P-T\,\sum_{i=1}^k n_i\Big)^2-
\frac{\pi}{9}\,\Big(\sum_{i=1}^k n_i\,q_i^2\Big)^3.
\end{equation}

Without loss of generality, assume that $q_1=0,$ $q_2>0,$ and $q_3<0.$ Then the
change of variables $\widetilde n_1=\sum_{i=1}^k n_i,$
$\widetilde n_2=\sum_{i=1}^k n_i\,q_i,$ 
$\widetilde n_3=(\pi/9)^{1/3}\,\sum_{i=1}^k n_i\,q_i^2,$ $\widetilde n_j=n_j,$ 
where $j=\overline{4,\,k},$ is nondegenerate and can be extended to a contact 
transformation. For this purpose, choose $\widetilde \mu_i$ that satisfy 
$\widetilde n_i\,\rmd\widetilde \mu_i=n_i\,\rmd\mu_i$ $(i=\overline{1,\,k}),$ 
and leave $P,\,T,$ and $s$ unchanged. Notice that condition~$(\ref{Eq16}_1)$ 
takes the form $\widetilde n_2=0.$

Let us denote
\begin{equation}\label{Eq18} 
\eqalign{
z_1=\widetilde\mu_1,\quad z_j=\widetilde\mu_{j+1},\quad z_k=T,\cr
z_{k+1}=\widetilde n_1,\quad z_{k+j}=\widetilde n_{j+1},\quad z_{2k}=s,\quad
z_{2k+1}=P}
\qquad(j=\overline{2,\,k-1}).
\end{equation}
In new variables, the Hamiltonian \eref{Eq17} is written
\begin{equation*}
H_0=z_k\,(z_{2k+1}-z_k\,z_{k+1})^2-z_{k+2}^3.
\end{equation*}

Inequality~$(\ref{Eq16}_2),$ which takes the form $z_{k+2}\ll z_k^3$ 
according to \eref{Eq18}, implies a natural condition on weight $\gamma$ of 
the variable $z:$ $\gamma_{k+2}=3\,\gamma_k.$ It follows in turn from 
\eref{Eq12} that $\gamma_{2k+1}=4\,\gamma_k.$ Furthermore, \eref{Eq18} 
implies $\gamma_{k+j}=\gamma_{k+1}$ $(j=\overline{2,\,k-1}).$ The remaining
components of $\gamma$ are determined by the canonicity condition. Thus, in 
the case under consideration, the natural canonical weight has the form 
$\gamma=(1,\,\ldots,\,1,\,3,\,\ldots,\,3,\,4),$ and $H_0$ is a 
quasi-homogeneous polynomial in $z$ of generalized degree $9$ with weight 
$\gamma.$

\begin{theorem}
The minimal resonant set $\mathfrak S$ corresponding to a lexicographical 
order determined by the order $z_{2k+1}\succ z_k\succ z_{k+1}\succ\ldots\succ 
z_{2k-1}$ consists of the following monomials:

\begin{enumerate}
\item $z_k^{i_k}\,\cdots\,z_{2k-1}^{i_{2k-1}};$
\item $z_{2k+1}\,z_k^{i_k+1}\,z_{k+j},$ where $j=\overline{1,\,k-1},$ and 
$z_{2k+1}^3\,z_k^{i_k+1};$  
\item $z_{2k+1}^{i_{2k+1}}\,z_{k+1}^{i_{k+1}}\,\cdots\,z_{2k-1}^{i_{2k-1}},$ 
except $z_{2k+1}\,z_{k+2}^3,$
\end{enumerate}
where $i_{2k+1},\,i_k,\,i_{k+1},\,\ldots,\,i_{2k-1}\in\mathbb Z_+.$
\end{theorem}

\begin{proof}
For all $i_{2k+1},\,i_k,\,i_{k+1},\,\ldots,\,i_{2k-1}\in\mathbb Z_+,$ we have
\begin{eqnarray}\fl
\{H_0,\,z_{2k+1}^{i_{2k+1}}\,z_k^{i_k}\,\cdots\,z_{2k-1}^{i_{2k-1}}\}=
z_{2k+1}^{i_{2k+1}-1}\,z_k^{i_k}\,\cdots\,z_{2k-1}^{i_{2k-1}}\times\nonumber\\
\times\Big[2\,\Big(\frac{i_{2k+1}}{2}+i_{k+1}+\ldots+i_{2k-1}-1\Big)\,
z_{2k+1}^2\,z_k-\label{Eq19}\\\quad-
2\,(i_{k+1}+\ldots+i_{2k-1}-1)\,z_{2k+1}\,z_k^2\,z_{k+1}-
i_{2k+1}\,z_k^3\,z_{k+1}^2+
2\,i_{2k+1}\,z_{k+2}^3\Big].\nonumber  
\end{eqnarray}
This implies that all monomials divisible by $z_{2k+1}\,z_k,$ except 
$z_{2k+1}\,z_k^{i_k+1}\,z_{k+j}$ for $j=\overline{1,\,k-1}$ and 
$z_{2k+1}^3\,z_k^{i_k+1},$ do not belong to $\mathfrak S.$ 

Furthermore, in case where $i_{2k+1}+2\,i_{k+1}+\ldots+2\,i_{2k-1}=2,$ the 
right-hand side of \eref{Eq19} is identically zero if $i_{2k+1}=0.$ Reducing 
one by one the leading monomials in the Poisson bracket \eref{Eq19} in the 
case where $i_{2k+1}=2$ and $i_{k+j}=0$ $(j=\overline{1,\,k-1}),$ we come to 
the following equality:
\begin{equation*}
\{H_0,\,z_k^{i_k}\,(z_{2k+1}-z_k\,z_{k+1})^2\}=
4\,z_k^{i_k}\,(z_{2k+1}\,z_{k+2}^3-\,z_k\,z_{k+1}\,z_{k+2}^3).
\end{equation*}
This implies that the monomial $z_{2k+1}\,z_{k+2}^3$ does not belong to 
$\mathfrak S.$ For $i_k\ge1,$ further reduction of the leading monomial gives 
identical zero.

Collecting all the results together, we come to the statement of the theorem.
\end{proof}

From now on, we consider, for simplicity, the plasma consisting of the 
hydrogen atoms, protons and electrons. We denote the corresponding parameteres 
by the subindices $\mathrm a,\,\mathrm p,$ and $\rme.$ In this case, 
$q_{\mathrm a}=0,$ $q_{\mathrm p}=-q_\rme=e,$ where $e$ is the elementary 
charge, $n_\rme=n_{\mathrm p},$ and
\begin{equation}
\eqalign{
z_1=\mu_{\mathrm a},\quad z_2=\frac{3^{2/3}}{2\,\pi^{1/3}\,e^2}\,
(\mu_{\mathrm p}+\mu_\rme-2\,\mu_{\mathrm a}),\quad z_3=T,\nonumber\\ 
z_4=n_{\mathrm a}+2\,n_{\mathrm p},\quad 
z_5=\frac{2\,\pi^{1/3}\,e^2}{3^{2/3}}\,n_{\mathrm p},\quad z_6=s,\quad 
z_7=P.}\label{Eq20}
\end{equation}

Theorem~4 gives the following generalized normal form of 
the equation of state:
\begin{eqnarray}\fl
z_3\,(z_7-z_3\,z_4)^2-z_5^3+
\sum_{i+3j+3k\ge10} h^{[0,i,j,k]}\,z_3^i\,z_4^j\,z_5^k+
\sum_{4i+3j+3k\ge10} h^{[i,0,j,k]}\,z_7^i\,z_4^j\,z_5^k+\nonumber\\+
\sum_{i=1}^\infty [h^{[3,i,0,0]}\,z_7^3+(h^{[1,i+2,1,0]}\,z_4+
h^{[1,i+2,0,1]}\,z_5)\,z_7\,z_3^2]\,z_3^i=0,\label{Eq21}
\end{eqnarray}
where $h^{[1,0,0,3]}=0.$

Let us consider the case, where the perturbation consists of elements of the 
minimal resonant set $\mathfrak S$ of generalized degree $10.$ According to 
\eref{Eq21} and the physical restriction $z_7/z_3\rightarrow z_4$ as either
$z_3\rightarrow\infty$ or $z_5\rightarrow0,$ the form of such perturbation is 
determined uniquely up to three constant coefficients
\begin{equation}\label{Eq22}
H_\gamma^{[10]}=-z_5\,(a\,z_4+b\,z_5+c\,z_3^3)\,(z_7-z_3\,z_4).
\end{equation}
Note that if we considered an arbitrary perturbation (not in normal form), we 
should have additionally taken into account the terms $z_7^2\,z_3^2$ and 
$z_7\,z_3^6.$ 

\bigskip
In order to find out the physical meaning of $a,\,b,$ and $c,$ first look at 
the behaviour of the Debye-H\"uckel plasma $(a,\,b,\,c=0)$ ignoring 
$(\ref{Eq16}_2).$ By integrating \eref{Eq14}, we come to the following 
expressions for $z_1$ and $z_2:$ 
\begin{equation*}
z_1=T\,\ln z_4+C_1,\quad z_2=
-3^{2/3}\,\pi^{1/6}\,e\,\sqrt{\frac{2\,\alpha\,z_4}{(1+\alpha)\,T}}+C_2,
\end{equation*}
where $\alpha=n_\mathrm{p}/(n_\mathrm{a}+n_\mathrm{p})$ is called the 
\textit{ionization coefficient}. Recall that $\alpha$ is an integral of 
\eref{Eq14}. Since at low values of the concentration $z_4,$ plasma turns to a 
mixture of ideal gases, $C_1$ and $C_2$ can be determined by replacing the 
chemical potentials of components in \eref{Eq20} with the chemical potentials 
of ideal gases that have the form 
\begin{equation*}
\mu_\mathrm{id}=T\,
\ln\Big[\frac{n}{Z}\,\Big(\frac{2\,\pi\,\hbar^2}{m\,T}\Big)^{3/2}\Big],
\end{equation*}
where $n$ denotes the concentration, $Z$ the partition function of a particle, 
and $m$ the particle mass \cite[p~163]{landau}. For electrons and protons 
$Z=2,$ and for the hydrogen atoms, for simplicity, we take 
$Z=4\,\exp(\mathrm{Ry}/T),$ where 
$\mathrm{Ry}\approx2.18\cdot10^{-11}\ \mathrm{erg}$ is the Rydberg constant. 
This partition function ignores all the energy levels of the hydrogen atom but 
the lowest one. For more precision the Planck-Larkin partition function should 
be used.

Solve the equation of chemical equilibrium that takes the form
\begin{equation}\label{Eq23}
z_1+3^{-2/3}\,2\,\pi^{1/3}\,e^2\,z_2=0
\end{equation} 
to obtain
\begin{equation}\label{Eq24}
\alpha=\Big[1+n_\mathrm{p}\,\Big(\frac{2\,\pi\,\hbar^2\,m_\mathrm{a}}{m_\rme\,
m_\mathrm{p}\,T}\Big)^{3/2}\,\exp\Big(\frac{\mathrm{Ry}}{T}-
\frac{e^3\,\sqrt{8\,\pi\,n_\mathrm{p}}}{T^{3/2}}\Big)\Big]^{-1},
\end{equation}
which implies the effect of the ionization potential decreasing as 
$n_\mathrm{p}$ increases. Note that this expression is equivalent to the 
classical Saha formula with the partition function of the hydrogen atom 
replaced by 
$Z=4\,\exp(\mathrm{Ry}/T-e^3\,\sqrt{8\,\pi\,n_\mathrm{p}}/T^{3/2}).$ Such 
approach to perturbation theory of plasmas consisting in perturbing the 
hydrogen atom partition function is usually called a \textit{chemical model} 
and used, for instance, in \cite{khomkin} where the Planck-Larkin and the 
nearest neighbour partition functions are used in this context (for further 
information on chemical model, see \cite[ss 5.1,\,5.2]{son}).

Denote $n_0=n_\mathrm{a}+n_\mathrm{p}$ and rewrite \eref{Eq12} as follows
\begin{equation}\label{Eq25}
P=T\,n_0\,(1+\alpha)-\frac{\sqrt{8\,\pi}\,e^3}{3\,T^{1/2}}\,
(\alpha\,n_0)^{3/2}.
\end{equation}
Then \eref{Eq24} and \eref{Eq25} together with $n_\mathrm{p}=\alpha\,n_0$ 
give a complete thermodynamical description of Debye-H\"uckel plasma.

According to \eref{Eq24}, in the low density region where 
$(\ref{Eq16}_2)$ is satisfied, the Saha approximation is valid. As the density 
$n_0$ increases the ionization coefficient stabilizes at first, and then after 
$n_0$ exceeds the value $\approx8\cdot10^{21}\ \mathrm{particles/cm^3},$ 
rapidly tends to $1$ (see \fref{fig1a}). At temperatures less than 
$\approx23\,000\ \mathrm K$ the curve $P(n_0)$ has two branches where 
$(\partial P/\partial n_0)_T>0$ (see \fref{fig1b}) which corresponds to 
splitting of the plasma into two stable phases. The thin one is weakly ionized 
and called the \textit{dielectric gas phase} and the dense is strongly ionized 
and called the \textit{metal gas phase}. This phenomenon is called the 
\textit{plasma phase transition}. For temperatures greater than 
$\approx23\,000\ \mathrm K$ metal gas phase becomes unstable since 
$(\partial P/\partial n_0)_T$ is negative on the corresponding branch. 
An overview of relevant theoretical results can be found in \cite[s 6.4]{son} 
and references therein. It should be noted that plasma phase transition has 
not yet been detected experimentally. For this reason it is often called 
hypothetical in literature.

As can be seen from \fref{fig1b}, the density $n_0$ is bounded above. Beyond 
this bound the pressure becomes negative, which is meaningless. Apparently, at 
higher densities a thermodynamical description of plasmas is impossible due to 
absence of quasi-static states which are the only considered in thermodynamics.

\begin{figure}
\subfloat[]{\fbox{\includegraphics[trim=3 0 3 0, scale=0.9]{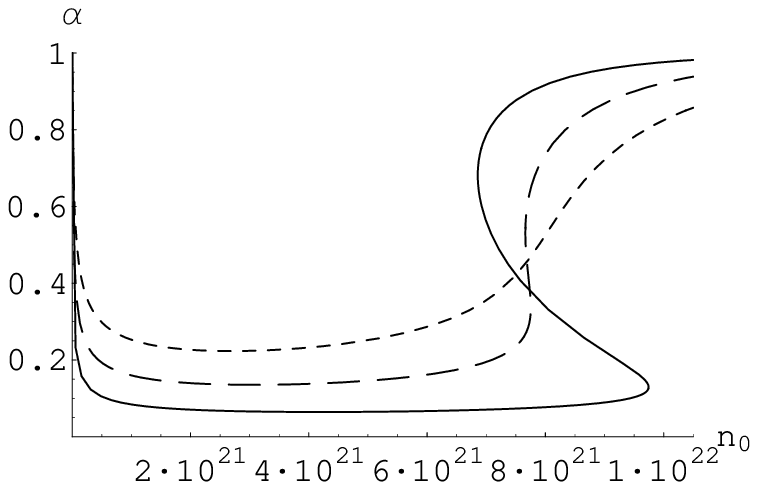}}
\label{fig1a}}\hfill
\subfloat[]{\fbox{\includegraphics[trim=3 0 3 0, scale=0.9]{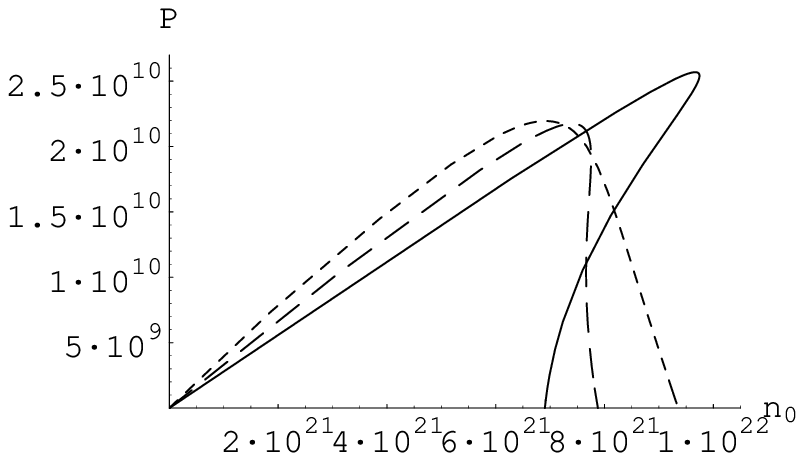}}
\label{fig1b}}
\caption{$\alpha(n_0)$ and $P(n_0)$ in the Debye-H\"uckel model at 
$T=20\,000\ \mathrm K$ (\full), $T=23\,000\ \mathrm K$ (\broken), and 
$T=26\,000\ \mathrm K$ (\dashed).}
\label{fig1}
\end{figure}

\bigskip
System \eref{Eq14} for perturbed Debye-H\"uckel plasmas with the perturbation 
\eref{Eq22} can be integrated explicitly as well. The results for small 
perturbations are given in \fref{fig2}. 
\begin{figure}[t]
\subfloat[]{\fbox{\includegraphics[trim=3 0 3 0, scale=0.9]{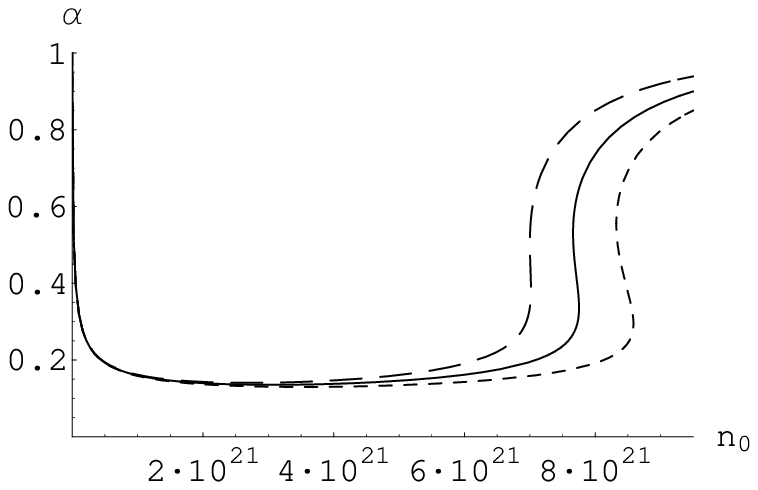}}
\label{fig2a}}\hfill
\subfloat[]{\fbox{\includegraphics[trim=3 0 3 0, scale=0.9]{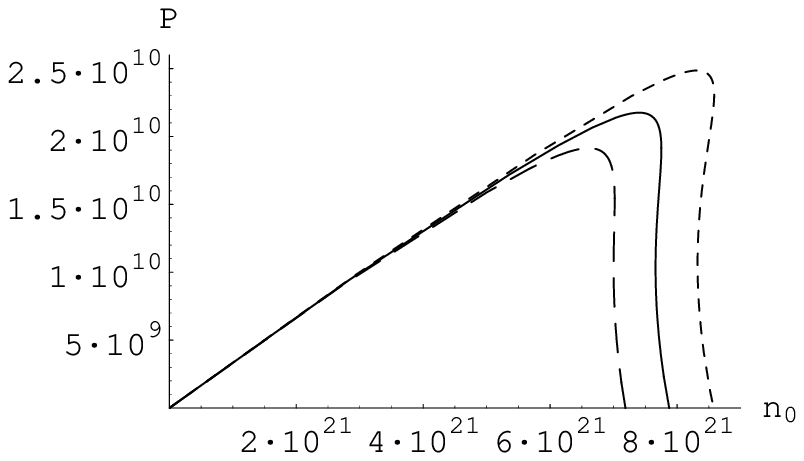}}
\label{fig2b}}\\
\subfloat[]{\fbox{\includegraphics[trim=3 0 3 0, scale=0.9]{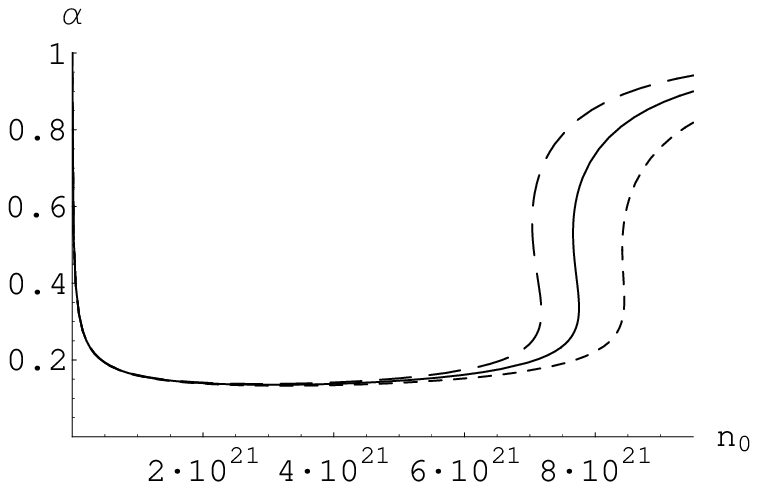}}
\label{fig2c}}\hfill
\subfloat[]{\fbox{\includegraphics[trim=3 0 3 0, scale=0.9]{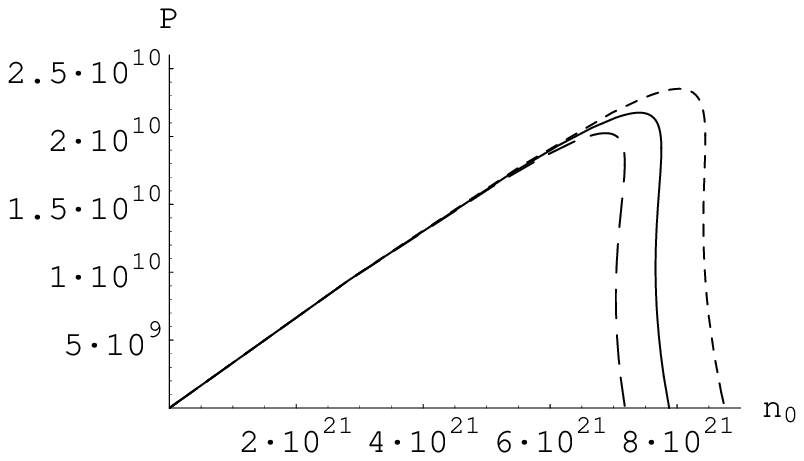}}
\label{fig2d}}\\
\subfloat[]{\fbox{\includegraphics[trim=3 0 3 0, scale=0.9]{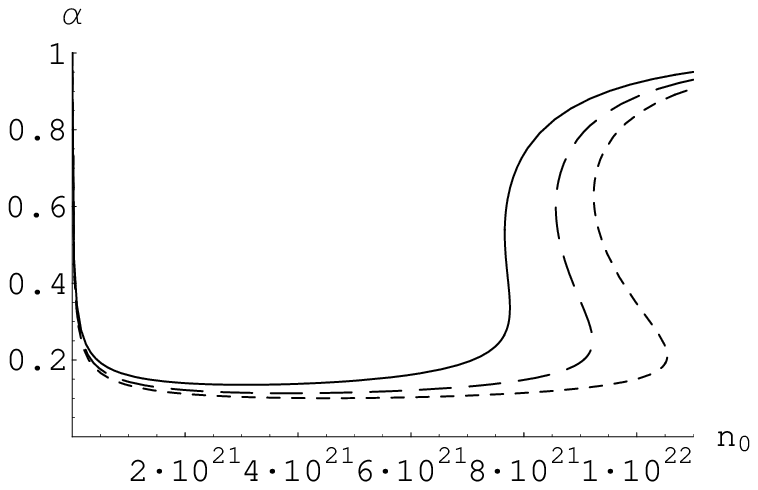}}
\label{fig2e}}\hfill
\subfloat[]{\fbox{\includegraphics[trim=3 0 3 0, scale=0.9]{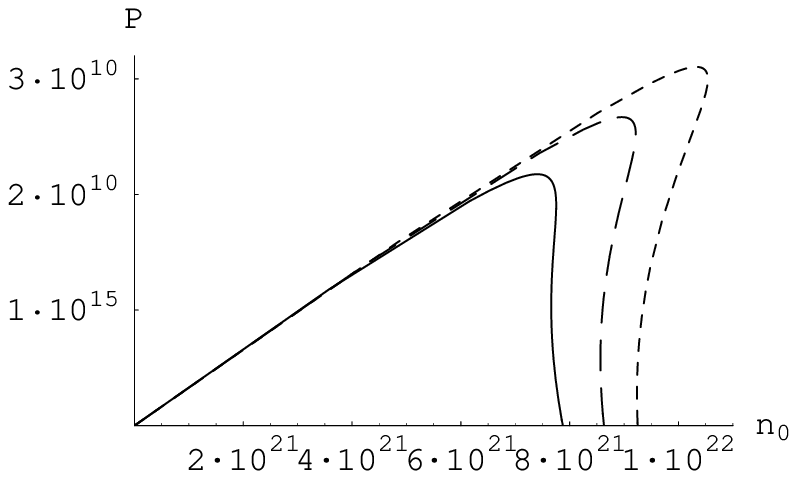}}
\label{fig2f}}
\caption{$\alpha(n_0)$ and $P(n_0)$ in perturbed Debye-H\"uckel plasmas at 
$T=23\,000\ \mathrm K$ for\\
(a)--(b): $a=0$ (\full), $a\,T/e^4\approx-0.04$ (\broken), and 
$a\,T/e^4\approx0.04$ (\dashed), $b,\,c=0;$\\
(c)--(d): $b=0$ (\full), $b\,T/e^2\approx-0.06.$ (\broken), and 
$b\,T/e^2\approx0.06.$ (\dashed), $a,\,c=0;$\\
(e)--(f): $c=0$ (\full), $c\,e^2\,T\approx0.035$ (\broken), and 
$c\,e^2\,T\approx0.07.$ (\dashed), $a,\,b=0.$}
\label{fig2}
\end{figure}
The coefficients $a$ and $b$ affect the pressure and the concentrations of 
phases at the point of phase transition without having a significant effect on 
stability of the dense phase (see figures \ref{fig2a}--\ref{fig2d}). The 
coefficient $a$ has a greater impact on properties of the thin phase while $b$ 
on the properties of the dense phase. At the same time the critical 
temperature increases with $c,$ which affects stability of the dense phase, as 
shown on figures~\ref{fig2e},\,\ref{fig2f}.

Our conclusions on influence of the coefficients $a,\,b$ and $c$ are based on 
the properties of graphs of $\alpha(n_0)$ and $P(n_0)$ at fixed temperatures.
Let us show that they are invariant under contact transformations. There are 
four monomials of generalized degree $5,$ namely, $z_7\,z_3,$ $z_3^5,$ 
$z_3^2\,z_4$ and $z_3^2\,z_5,$ but only two of them give non-zero Poisson 
brackets with $H_0$ of generalized degree $10,$ 
\begin{equation*}
\{z_3\,z_7,\,H_0\}=-3\,z_3\,z_5^3+2\,z_3^2\,(z_7-z_3\,z_4)^2,\quad
\{z_3^5,\,H_0\}=2\,z_3^6\,(z_7-z_3\,z_4).
\end{equation*}
The corresponding 1-parameter groups of contact transformations have the form
\begin{eqnarray*}
&[z_4\rightarrow\rme^{\lambda\,z_3}\,z_4,\quad
z_5\rightarrow\rme^{\lambda\,z_3}\,z_5,\quad
z_6\rightarrow\rme^{\lambda\,z_3}\,(z_6+\lambda\,z_7),\quad
z_7\rightarrow\rme^{\lambda\,z_3}\,z_7],\\
&[z_6\rightarrow z_6+5\,\lambda\,z_3^4,\quad
z_7\rightarrow z_7+\lambda\,z_3^5],
\end{eqnarray*}
where the omitted variables are left unchanged, and $\lambda$ is a parameter. 
Both transformations preserve the equation of chemical equilibrium 
\eref{Eq23}, since 
they do not change the chemical potentials. The first transformation 
represents the uniform scaling of pressure and concentrations by the scale 
factor $\rme^{\lambda\,z_3},$ so it does not affect our considerations. The 
second transformation shifts pressure by the value $\lambda\,z_3^5,$ which 
leads to violation of equality of $P$ to zero at zero density and hence has no 
physical sence, so the term $z_3^6\,(z_7-z_3\,z_4)$ is restricted in any 
perturbation of the Debye-H\"uckel model. 

\section{Conclusion}

As can be seen from the above examples, the introduced method of generalized 
normal forms is a flexible instrument for studying the critical phenomena in 
various physical systems. It can be used whenever there is a natural 
symplectic or contact form. To our knowledge, thermodynamic equations of state 
have never been considered from this point of view. In the context of 
infinitesimal contact transformations and the corresponding systems of ODEs, 
normal forms were previously studied only in the case of quasi-homogeneous 
unperturbed part $H_0$ of generalized degree $2$ with weight 
$\gamma=(1,\,\ldots,\,1,\,2)$ with generic coefficients 
(see \cite[p~37]{arnold} and references therein).

In both examples the unperturbed Hamiltonian $H_0$ was a quasi-homogeneous 
polynomial with a canonical weight, and a generalized normal form was obtained 
rather easily. However, it is not so for all interesting systems. For example, 
the Hamiltonian
\begin{equation*}
H_0=\frac{m\,p_1^2+\lambda_1\,q_1^{2m}}{2\,m}-
\frac{n\,p_2^2+\lambda_2\,q_2^{2n}}{2\,n},
\end{equation*}
where $\lambda_1,\,\lambda_2>0,$ $m,\,n\in\mathbb N,$ $n>m>1,$ is not 
quasi-homogeneous for any canonical weight. Such a Hamiltonian arises in 
studying the interaction between two nonlinear oscillators at the level of 
perturbation. The problem of stability of the equilibrium for a system with 
such an unperturbed part was partially solved in \cite{bibikov} by methods of 
KAM theory: for sufficiently high orders of the perturbation, the origin is 
Lyapunov stable. At low orders it is natural to deal with a normal form 
instead of an arbitrary perturbation, since it contains a significantly 
smaller number of terms. In the general case, explicit formulas for minimal 
resonant sets turn out to be very complicated. Nevertheless, Lemma~1 gives an 
algorithm, which can be used to obtain the leading terms of the 
Poisson brackets, and hence a minimal resonant set, for given $m,\,n,\,N,$ and 
linear order on monomials. For example, the minimal resonant set for $m=2,$ 
$n=3,$ $N=11,$ and a lexicographical order determined by the order 
$p_1\succ q_1\succ p_2\succ q_2$ consists of monomials 
$p_2^{i_2}\,q_1^{j_1}\,q_2^{j_2},$ where $7\le j_1+i_2+j_2\le11$ and 
$j_1\ge i_2,$ except the following: 
$q_1^3\,q_2^6,$ $q_1^3\,q_2^7,$ $q_1^3\,q_2^8,$ $p_2\,q_1^3\,q_2^5,$ 
$p_2\,q_1^3\,q_2^6,$ $p_2\,q_1^3\,q_2^7,$ $p_2^2\,q_1^3\,q_2^4,$ 
$p_2^2\,q_1^3\,q_2^5,$ $p_2^2\,q_1^3\,q_2^6,$ $p_2^3\,q_1^3\,q_2^3,$ 
$p_2^3\,q_1^3\,q_2^4,$ $p_2^3\,q_1^3\,q_2^5,$ $p_2\,q_1^4\,q_2^4,$ 
$p_2\,q_1^4\,q_2^5,$ $p_2\,q_1^4\,q_2^6,$ $p_2^2\,q_1^4\,q_2^3,$ 
$p_2^2\,q_1^4\,q_2^4,$ $p_2^2\,q_1^4\,q_2^5,$ $p_2^3\,q_1^4\,q_2^2,$ 
$p_2^3\,q_1^4\,q_2^3,$ $p_2^3\,q_1^4\,q_2^4,$ $p_2^4\,q_1^4\,q_2,$ 
$p_2^4\,q_1^4\,q_2^2,$ $p_2^4\,q_1^4\,q_2^3,$ $p_2^2\,q_1^5\,q_2^3,$ 
$p_2^2\,q_1^5\,q_2^4,$ $p_2^3\,q_1^5\,q_2^2,$ $p_2^3\,q_1^5\,q_2^3,$ 
$p_2^4\,q_1^5\,q_2,$ $p_2^4\,q_1^5\,q_2^2,$ $p_2^5\,q_1^5,$ 
$p_2^5\,q_1^5\,q_2,$ $p_2^3\,q_1^6\,q_2,$ $p_2^3\,q_1^6\,q_2^2,$ 
$p_2^4\,q_1^6,$ $p_2^4\,q_1^6\,q_2,$ $p_2^5\,q_1^6,$ $q_1^7\,q_2^4,$ 
$p_2\,q_1^7\,q_2^3,$ $p_2^2\,q_1^7\,q_2^2,$ $p_2^3\,q_1^7\,q_2,$ 
$p_2^4\,q_1^7,$ $p_2\,q_1^8\,q_2^2,$ $p_2^2\,q_1^8\,q_2,$ $p_2^3\,q_1^8,$ 
$p_2^2\,q_1^9.$ This set consists of $134$ monomials, while in an arbitrary 
perturbation there are $1155$ terms at these degrees.  

\ack

This research is supported by the Chebyshev Laboratory (Department of
Mathematics and Mechanics, Saint-Petersburg State University) under RF 
Government grant~11.G34.31.0026. The author is grateful to 
Dr. S.G.~Kryzhevich, Dr. V.V.~Basov, and Prof. Yu.N.~Bibikov (Department of 
Differential Equations, Mathematics and Mechanics Faculty, Saint-Petersburg 
State University) for valuable comments.


\section*{References}

\begin{thebibliography}{99}

\bibitem{kobayasi}
Kobayashi S 1995 \textit{Transformation Groups in Differential Geometry} 
(Berlin: Springer, reprint of the 1972 edition). 

\bibitem{bruno}
Bruno A D 1988 The normal form of a Hamiltonian system 
\textit{Russian Mathematical Surveys} \textbf{43}(1) 25--66.

\bibitem{meyer}
Meyer K R 1974 Normal forms for Hamiltonian systems 
\textit{Celestial Mechanics} \textbf{9}(4) 517--22.

\bibitem{belitsky} 
Belitsky G R 1979 Invariant normal forms of formal series 
\textit{Functional Analysis and Applications} \textbf{13} 46--67.

\bibitem{BV} 
Basov V V and Vaganyan A S 2010 Normal Forms of Hamiltonian Systems 
\textit{Differential Equations and Control Processes} 2010-4 86--107 
(Available at http://www.math.spbu.ru/diffjournal).

\bibitem{belitsky77} 
Belitsky G R 1977 Normal forms relative to the action of a group on a space 
\textit{Mathematics of the USSR-Izvestiya} \textbf{11}(5) 1001--10.

\bibitem{landau} 
Landau L D and Lifshitz E M 2002 \textit{Course of Theoretical Physics vol 5
Statistical Physics Part I} 5th~ed (Moscow: Fizmatlit, in Russian).

\bibitem{fortov} 
Bushman A V and Fortov V E 1983 Model equations of state 
\textit{Sov. Phys. Usp.} \textbf{26} 465--96.

\bibitem{krichevsky} 
Krichevsky I R 1970 \textit{Concepts and Fundamentals of Thermodynamics} 
(Moscow: Khimiya, in Russian).

\bibitem{feynman} 
Feynman R P 1982 \textit{Statistical Mechanics: A Set of Lectures} 
(Reading, Massachusetts: The Benjamin/Cummings Publishing Co.).

\bibitem{khomkin} 
Khomkin A L and Shumikhin A S 2008 Features of the chemical models of a 
nonideal atomic plasma at high temperatures \textit{Plasma Physics Reports}, 
\textbf{34}(3) 251--56.

\bibitem{son} 
Iosilevsky I L, Krasnikov Yu G, Son E E and Fortov V E 2000 
\textit{Thermodynamics and transport in non-ideal plasmas} (Moscow: MIPT
Publishing, in Russian).

\bibitem{arnold}
Arnold V I  2002 \textit{Geometrical Methods in the Theory of Ordinary 
Differential Equations} (Moscow: Regulyarnaya i Khaoticheskaya Dinamika, in 
Russian).

\bibitem{bibikov} 
Bibikov Yu N 2010 On the stability of the state of equilibrium of some 
essentially non-linear Hamiltonian systems with two degrees of freedom
\textit{Differential Equations and Control Processes} 2010-4 26--34 
(Available at http://www.math.spbu.ru/diffjournal, in Russian).
\end{thebibliography}
\end{document}